\documentclass{article}
\usepackage{mathtools,blkarray}
\usepackage[english, activeacute]{babel} 
\usepackage{mathtools}
\usepackage{titling}
\usepackage{marvosym}
\usepackage{geometry}
\usepackage{lipsum}
\usepackage[utf8]{inputenc}
\usepackage{listings}
\usepackage{adjustbox}
\usepackage{multicol}
\usepackage{tikz-cd}
\usetikzlibrary{positioning}
\tikzset{main node/.style={circle,fill=blue!20,draw,minimum size=0.1cm,inner sep=0pt}}
\usetikzlibrary{calc}
\usepackage{graphicx}
\usepackage{amssymb}
\usepackage{qtree}
\usepackage{tipa}
\usepackage{hyperref}
\usepackage{amsthm}
\usepackage{amsmath}
\allowdisplaybreaks
\usepackage{makeidx}
\usepackage{color}
\usepackage{enumerate}
\usepackage{doi}
\usepackage{mathdots}
\usepackage{xcolor}
\usepackage[backend=biber,style=alphabetic,sorting=ynt]{biblatex}
\renewbibmacro{in:}{}
\usepackage{csquotes}
\usepackage[all,cmtip]{xy}

\usepackage{etoolbox}

\numberwithin{figure}{section}
\numberwithin{equation}{section}

\AtBeginEnvironment{figure}{\stepcounter{equation}}
\AtBeginEnvironment{equation}{\stepcounter{figure}}

\newcommand{\comm}[1]{}

\newtheorem{teorema}{Theorem}
\newtheorem{conjetura}[teorema]{Conjecture}

\newtheorem{proposicion}[teorema]{Proposition}
\newtheorem{corolario}[teorema]{Corollary}

\theoremstyle{definition}
\newtheorem{definicion}[teorema]{Definition}

\newtheorem{cuestion}[teorema]{Question}

\newtheorem{construccion}[teorema]{Construction}

\newtheorem{ejemplo}[teorema]{Example}

\newtheorem{remark}[teorema]{Remark}

\DeclareMathOperator{\dist}{dist}

\DeclareMathOperator{\Sym}{Sym}

\DeclareMathOperator{\rcs}{rcs}

\DeclareMathOperator{\Ima}{Im}

\DeclareMathOperator{\diag}{diag}

\DeclareMathOperator{\rt}{rt}

\DeclareMathOperator{\coeff}{coeff}
\DeclareMathOperator{\vol}{vol}

\bibliography{references}

\begin{document}
\title{The generalized Lax conjecture is true for topological reasons related to compactness, convexity and determinantal deformations of increasing products of pointwise approximating linear forms}
\author{Alejandro Gonz\'alez Nevado\thanks{Alejandromagyar@gmail.com}}
\date{January 2026}
\maketitle

\begin{abstract}
We develop a topological approach to prove the generalized Lax conjecture using the fact that determinants of sufficiently big symmetric linear pencils are able to express the rigidly convex sets of RZ polynomials of any degree $d$. Monicity of the representation is assessed through a topological argument that allows us to perturbate a sufficiently close linear approximation into a suitable nice determinantal multiple of the initial RZ polynomial with the same rigidly convex set. The perturbation can be smoothly performed. This fact is what will allow us to determine that the multiple obtained respects the initial rigidly convex sets. This argument provides thus a full proof of the generalized Lax conjecture. However, an effective proof providing the representation in nice terms seems far from reachable at this moment.
\end{abstract}

\section{Introduction}
The generalized Lax conjecture is one of the most important open problems at the intersection of mathematical optimization and real algebraic geometry. It deals with the representation of sets delimited at the same time by an algebraic (monic symmetric determinantality) and a geometric property (keeping rigidly convex set). These sets are defined by a border given by the set of real zeros around the origin of a polynomial having only real zeros across each real line through the origin. The exact definition is as follows.

\begin{definicion}[RZ polynomials and RCSs] Let $p\in\mathbb{R}[\mathbf{x}]=\mathbb{R}[x_{1},\dots,x_{n}]$ be a polynomial with real coefficients. We say that $p$ is \textit{real zero} if, for each $a\in\mathbb{R}^{n}$, we have that the univariate polynomial $p(ta)$ has all its roots real. Let $V\subseteq\mathbb{R}^{n}$ be the set of real zeros of $p$. We call the set $\rcs(p)$ given by the Euclidean closure of the connected component of $\mathbf{0}$ in $\mathbb{R}^{n}\smallsetminus V$ the \textit{rigidly convex set} of $p$.
\end{definicion}

A classic example of such polynomials are determinants of MSLMP (monic symmetric linear matrix polynomials). It is know that, in the case of two variables, these are actually all the polynomials verifying that property.

\begin{teorema}[Helton-Vinnikov]\cite[Theorem 2.2]{helton2007linear}
Let $p\in\mathbb{R}[x_{1},x_{2}]$ be a RZ polynomial of degree $d$ with $p(\mathbf{0})=1$. Then there exist real symmetric matrices $A_{1}$ and $A_{2}$ of size $d$ such that $p=\det(I_{d}+x_{1}A_{1}+x_{2}A_{2}).$ And the entries of these matrices can be generically explicitly given by continuous functions of the coefficients.
\end{teorema}

The generalized Lax conjecture asks also loosely if something like this happens for more variables. The immediate and naive generalization of such result is not true by dimensionality reasons. It is already well known that, for dimensionality reasons, there are more RZ polynomials than there are determinantal polynomials when $n>2$. 

\begin{teorema}[Dimensionality obstruction]\cite[Theorem]{nuij1968note}\label{dime}
The space of RZ polynomials in $n$ variables of degree $d$ contains an open inside the space of polynomials in the same number of variables and degree $\mathbb{R}[\mathbf{x}]_{\leq d}$. In fact, every such RZ polynomial $p\in\mathbb{R}[\mathbf{x}]_{\leq d}$ can be expressed as the limit of smooth real zero polynomials of the same degree and number of variables. The dimension of the space $\mathbb{R}[\mathbf{x}]_{\leq d}$ is, for $n>2$ and $d>1$, bigger than the amount of (independent) entries in any possible MSLMP representation of these polynomials.
\end{teorema}

The cited theorem of Nuij says even more about the topology of this space. We copy the theorem adapting it to the environment of RZ polynomials directly.

\begin{teorema}[Simplicity as topological space]\label{simpli} The topological space of RZ polynomials verifies the following properties.
\begin{enumerate}\label{nuijth}
  \item The space of smooth RZ polynomials is open.
  \item Every RZ polynomial is the limit of smooth RZ polynomials.
  \item The space of normalized ($p(\mathbf{0})=1$) RZ polynomials is connected and simply connected.
\end{enumerate}
\end{teorema}

We cannot have a determinantal representation because of the dimensionality obstruction presented in Theorem \ref{dime}. However, the requirement of a determinantal representation can be softened in order to try to produce a positive result. Several softenings have been shown to have exceptions and therefore are not true in general. Other softenings are well known to be true but the conditions are so weak that the original rigidly convex set could be destroyed along the way. For this reason, the generalized Lax conjecture (GLC) sits in the middle of these. The conjecture that we are going to prove here is the next softening of having a MSLMP determinantal representation.

\begin{conjetura}[Generalized Lax conjecture]\cite[Conjecture 1.5]{amini2019spectrahedrality}
Let $p\in\mathbb{R}[\mathbf{x}]$ be a RZ polynomial with $p(\mathbf{0})=1$. Then there exist another RZ polynomial $q\in\mathbb{R}[\mathbf{x}]$ with $\rcs(p)\subseteq\rcs(q)$ and symmetric matrices $A_{1},\dots, A_{n}$ such that $qp=\det(I+x_{1}A_{1}+\cdots+x_{n}A_{n})$.
\end{conjetura}

As we can see, the algebraic part is softened as much as possible while the geometric requirement of keeping the original rigidly convex set untouched is required. In this way, we keep a LMI representation of $\rcs(p)$, which is another way of phrasing the GLC.

\begin{conjetura}[GLC]
Every rigidly convex set is spectrahedral, i.e., admits a representation as a MSLMI of the form $\rcs(p)=\{\mathbf{x}\in\mathbb{R}^{n}\mid I+x_{1}A_{1}+\cdots+x_{n}A_{n} \mbox{\ is PSD}\}.$
\end{conjetura}

This conjecture has been proved in several instances of polynomials having further structure. See \cite{amini2019spectrahedrality}, for some examples. This includes cases for some well-known polynomials coming from combinatorics. However, it seems difficult to obtain general algebraic effective results for polynomials not associated to combinatorial objects. This is something one could expect. The form of the entries of the matrices in the LMI could be given in terms of very wild algebraic functions of the coefficients and there is no reason believe that we know much about these functions. Thus, this general algebraic effectivity seems like a too strong requirement for three reasons. We detail here which are these reasons and what makes us think that following general effectivity is a too complicated task to pursue upfront.

\subsection{Size explosion}

It is well-known that the size of the matrices involved in the possible monic symmetric determinantal representation of an arbitrary RZ polynomial grows \textit{very fast}. This implies that any effective algebraic description would have to deal with too much information and therefore being extremely difficult to obtain.

In particular, the accelerated growth we mention comes in the form of a theorem that bounds our problem from below. This means that the situation could in general be even much worse. 

\begin{remark}[Compact RCS]\label{compactrcs}
For the next theorem (and in general in the future unless we say the opposite), we will suppose wlog that our RZ polynomials have compact RCSs.
\end{remark}

\begin{teorema}[Uncontrollable growth of size of the LMP]\cite[Theorem 1.1]{raghavendra2019exponential}\label{uncon}
There exists an absolute constant $\kappa>0$ such that for all sufficiently large $n,d$, there exists a RZ polynomial $p\in\mathbb{R}[\mathbf{x}]$ of degree $d$ in $n$ variables whose rigidly convex set $\rcs(p)$ does not admit an $\eta$-approximate MSLMI representation of size $B\leq(n/d)^{\kappa d}$ for $\eta=1/n^{4nd}.$
\end{teorema}

As approximations will be central for our work here, we also remind the reader of what kind of approximations the result above talks about.

\begin{definicion}[Approximation]\label{defapr}
Let $(X,d)$ be a metric space and $A\subseteq X$ bounded. For $x\in X$ we define the \textit{radius map} $r(x,A)=\inf\{d(x,a)\mid a\in A\}$. Then, for closed bounded sets $A,B\subseteq X$, we define the \textit{Hausdorff distance} between two closed sets $h(A,B)=\max\{\sup_{a\in A}r(a,B),\sup_{b\in B}r(A,b)\}.$
\end{definicion}

It will be important to observe how these approximation look like in order to understand the idea of the proof that we are going to present here. The next remark will make things clearer.

\begin{remark}\label{remapr}
The Hausdorff distance measures therefore the maximal distance reachable between a point in $A$ and a point in $B$. Thus, the theorem claims that,  for $\eta=1/n^{4nd}$, there exist a $p$, with number of variables $n$ and degree $d$ sufficiently large, such that its rigidly convex set $\rcs(p)$ cannot be a approximate within distance $\eta$ by a spectrahedron given by a MSLMI of size $B\leq(n/d)^{\kappa d}$.
\end{remark}

The form of these approximation will have a fundamental impact in the future. However, we will require (the border of) our approximations to be very close not just to the rigidly convex set but to the whole set of real zeros that determines the RZ polynomial $p$.

\subsection{The topological approach}

The required growth of the representation invites us to look beyond algebraic directions. Topological methods have proven useful to deal with these questions. For example, see the topological discussions in \cite{helton2007linear}. Hence, our arguments here make use mainly of topological, continuity, convexity, smoothness and metric tools in order to produce a satisfactory proof of the generalized Lax conjecture.

\begin{remark}[Tools for a proof]
Topological considerations will allow us to make use of compactness to ensure that our representations lie within a tolerable distance from our original polynomials. We will see this compactness in particular from the point of view of the metric space of polynomials, where the set of RZ polynomial sits. This compactness will allow us to construct diagonal matrices of finite size that lie extremely close to our polynomials in terms of the metric we use. Finally, smoothness and convexity will allow us to perform a matrix deformation, changing the entries, until we reach our representation in a way that respects the original RCS, as the conjecture asks. These are therefore the tools that will play a big role here.
\end{remark}

However, these concepts will still need to be stringed together in a coherent way through many different appearances of the phenomena they describe in different places and realms at the same time. This work will require us to carefully distinguish where and when each of these approaches appear because they will appear several times with different faces and touching different objects seemingly deeply related.

\begin{remark}[Two views of continuity]\label{twoviews}
When we make a small perturbation of an RZ polynomial in terms of its coefficients, this translates to a small perturbation in terms of its roots. Thus, we can see continuity and perturbations through two different lenses: we can move around the zero set of the polynomial polynomially (meaning that what we get is still a polynomial) or we can perturb the coefficients. When one is sufficiently small so is the other. We will use this translation between perturbations at length during the proof.
\end{remark}

In a similar fashion, it will be important to consider the next remark about compactness and the space we have to look at.

\begin{remark}[Compactness viewed from two sides]\label{compactnessviews}
Taking a compact set in the space of RZ polynomials will amount to taking a bounded and closed set. We can again see this set in terms of the coefficients as an usual compact inside a metric space or we can see this properties through the lenses of the zero sets of the polynomials in this set. Looking at the zero sets, this compact set will have the form of a closed set of ovaloids that lie within a bounded distance from the initial ovaloids. All this, obviously, polynomially, which means that we stay the whole time within the range of what is expressible in polynomial terms.
\end{remark}

This encourages us to divide and partition the proof into different sections that will allow us to get deeper into the phenomena we are analyzing each time. Looking at the perturbations in the coefficients and the perturbations in the roots. This slow process will allow us to introduce some visualizations, examples and counterexamples that will make our journey around this landscape more satisfactory in the path of understanding the environment we are working in.

\begin{remark}[Slowly cooking a proof]
We will make clear how counterexamples found through other approaches have fed our research. We will also make an exercise of bookkeeping the transactions between the two points of views presented above: the forms of the perturbations of the coefficients and the roots.
\end{remark}

We hope that such treatment plays the role of making the reader aware of the many subtleties one has to deal with during our journey and, at the same time, of how much is yet to explore within these topics considered as areas of research just by themselves. Finally, all these explorations will come together to form the \text{obvious} proof of the GLC that we promise in the title.

\begin{remark}[Along the path]
In particular, along this path we will realize the importance of developing approximative methods and algorithms in order to approximate faster these rigidly convex sets. We will not expand in this direction here but this an obvious path to follow if one wants to achieve smaller representations, as we will not care here about how big these can eventually get. 
\end{remark}

When we deal with multivariate polynomials several monomials can have the highest degree, but there is just one monomial of degree $0$. This is why we prefer to center our attention on that monomial instead.

\begin{definicion}
We say that a polynomial is \textit{comonic} if its independent term is $1$.
\end{definicion}

In what follows, we are going to fix a RZ polynomial $p$ which is comonic and generic (and therefore smooth). We will require $p$ to have also compact real zero set $V$. We see that these conditions do not affect the generality of our arguments.

\begin{proposicion}\label{wlogcom}
If the GLC is true for all smooth RZ polynomials with compact real zero set, then it is true for all RZ polynomials.
\end{proposicion}

\begin{proof}
Suppose $p\in\mathbb{R}[x_{1},\dots,x_{n}]$ of even degree $d=2k$ has non-compact real zero set. Then we can homogenize it to \begin{gather*}
  P(X_{0},\dots, X_{n})=X_{0}^{d}p(\frac{X_{1}}{X_{0}},\dots, \frac{X_{n}}{X_{0}})
\end{gather*} in order to consider its behaviour at infinity. Then $P$ is a hyperbolic polynomial. We identify as usual the $n$-dimensional real projective space $\mathbb{P}^{n}(\mathbb{R})$ with the union of $\mathbb{R}^{n}$ and the hyperplane at infinity given by $X_{0}=0$. In this case, the affine coordinates $x$ and the projective coordinates $X$ are related by $x_{i}=\frac{X_{i}}{X_{0}}$. In this environment, projective real algebraic hypersurface $P(X)=0$ in $\mathbb{P}^{n}(\mathbb{R})$ is the projective closure of the affine real algebraic hypersurface $p(x)=0$ in $\mathbb{R}^{n}$. As $d$ is even, \cite[Theorem 5.2.a]{helton2007linear} tells us that the the porjective hypersurface of real zeros is the disjoint union of $k$ ovaloids $W_{1},\dots,W_{k}$ with $W_{i}$ in the interior of $W_{i+1}$ for $i\in[k-1]$ and $x^{0}$ lying in the interior of $W_{1}$. Therefore the fact that $p$ has non-compact real zero set comes from choosing the hyperplane at infinity $X_{0}=0$ cutting through these ovaloids. Choosing a different hyperplane at infinity, we can build another RZ polynomial $p'$ with compact real zero set (it is enough if we now choose the hyperplane at infinity to not cut the outermost ovaloid $W_{k}$) and it is obvious that a spectrahedral representation for $\rcs(p')$ translates to a spectrahedral representation for $\rcs(p)$. If $p$ has odd degree, it can be perturbed to a RZ polynomial of even degree (as a polynomial of degree $d$ can be seen as a polynomial of degree $d+1$ with an extra branch of zeros at infinity which can then be perturbed adequately) and we proceed similarly. A limit argument covers now both the odd case and the non-smooth case, finishing the proof.
\end{proof}

Continuing with our conventions, we will fix that the degree of $p$ is $d$ and it has $n$ \textit{meaningful} variables. Thus $p$ is a monic generic smooth RZ element of $\mathbb{R}[\mathbf{x}]_{d}$ with $\mathbf{x}=(x_{1},\dots,x_{n}).$ As we want it to have a compact RZ set, it will moreover need to have even degree $d=2k.$ In general, because of this, the number of variables will not play a very deep role while the degree will in fact be a bit more involved in our arguments. We will explain the terms used here further in the next mandatory section. We also make clear that, through these conditions, we are not really losing any generality, just gaining clarity.

\begin{remark}[Without loss of generality]
The polynomial $p$ can always be made \textbf{comonic} by a simple division as long as $p(\mathbf{0})\neq0$, which is something that every RZ polynomials verifies. We want that $p$ has \textbf{compact} real zero set because we will use the compactness of these curves in order to build determinantal representations. But this is not a real problem, as we can always compactify our environment passing to the projective setting and speaking about hyperbolic polynomials. Everything can be translated into that environment in an immediate way, as we saw in Proposition \ref{wlogcom}. If $p$ is not \textbf{smooth} or not having \textbf{even} degree, there will always be a smooth even degree perturbation close enough to it and this small perturbation will fall so close that any argument will carry over by continuity. The same applies to \textbf{genericity}.
\end{remark}

Anyway, at the end we will see how to modify the arguments in order to cover these limit cases when they appear. The message is that our assumptions are not problematic, just clarifying. They serve the only purpose of increasing clarity along our path here and easing our ability to visualize our arguments.

\section{Preliminaries}

Real zero polynomials generalize real rooted polynomials to the multivariate setting. They do this requiring real-rootedness of the restrictions along all the lines passing through the origin. Another way of defining real rooted polynomial is through how they split completely. This applies to RZ polynomials.

\begin{definicion}[Splitting definition of RZ polynomials]
Let $p\in\mathbb{R}[\mathbf{x}]$ be a polynomial with real coefficients. We say that $p$ is \textit{real zero} if, for each $a\in\mathbb{R}^{n}$, we have that the univariate polynomial $p(ta)$ splits in $\mathbb{R}$ in terms of linear factors. That means that we can write $p(ta)=c_{a}(t-r_{a,1})\cdots(t-r_{a,d})$ for $d$ the degree of $p(ta)$ in the variable $t$.
\end{definicion}

When, for some direction $a$, the degree changes it can only drop down the degree of $p$. This just means that some root has \textit{jumped} to infinity and therefore it disappears from our affine plane. This roots would be conserved in the projective setting using hyperbolic polynomials. However, as we will consider only polynomials with compact rigidly convex sets, this will not be a problem for us. This splitting also tells us something well-known in the case of real rooted polynomials.

\begin{teorema}[Splitting determines real rooted polynomials]
Let $p\in\mathbb{R}[x]$ be a (co)monic real rooted polynomial. Then its real roots totally determine $p$.
\end{teorema}

This means that, if we know all the real roots of a real rooted polynomial, then we can totally reconstruct the polynomial (up to a constant). Notice that this is not true if some roots are not real. That is, if there are non-real roots, the real roots do not completely determine the polynomial.

\begin{ejemplo}
Consider the polynomials $p=(x-1)(x-i)(x+i), q=(x-1)(x-2i)(x+2i)\in\mathbb{R}[x]$. They have the same real roots but differ in their complex roots and this provokes that they are different.
\end{ejemplo}

This a very easy example. The same thing happens in more variables. If we want to determine a multivariate polynomial by their real roots, we need that all their roots through any line are real so that we can reconstruct the value of the polynomial line by line adequately through all the real lines. In fact, we have the following result.

\begin{proposicion}\label{comonicreal}
If two comonic polynomials $p,q\in\mathbb{R}[\mathbf{x}]$ of degree $d$ have $d$ real roots along (almost) each line through the origin and they coincide, then $p=q$.
\end{proposicion}

\begin{proof}
Equality of the roots on (almost) each real line implies equality of the restriction to such lines. As this happens along (almost) all the lines this, in turn, implies equality along all of $\mathbb{R}^{n}$ (by continuity). Equality on $\mathbb{R}^{n}$ implies forma equality of polynomials. This is well known. See, for example, \cite{4942793}.
\end{proof}

Notice that the same is not true if there are complex roots. Easy examples show this.

\begin{ejemplo}
The comonic polynomials obtained after normalizing $p=(1-x_{1}-x_{2})(x_{1}-i)(x_{1}+i),q=(1-x_{1}-x_{2})(x_{1}-x_{2}-i)(x_{1}-x_{2}+i)\in\mathbb{R}[\mathbf{x}]$ are different but have the same real roots.
\end{ejemplo}

Thus, when we consider just polynomial with all their roots along real lines real (RZ polynomials), we do not even have to care about what happens outside of the reals. This means that, for these polynomials, we can restrict the reach of all the tools we need to those having a nice behaviour in the reals, nothing beyond that matters in our journey of understanding them.

\begin{remark}[Real is all you need]
For RZ polynomials, understanding their behaviour over the reals amounts to actually understanding the polynomial itself formally.
\end{remark}

Notice that we deal with all polynomials without taking care of them being reducible or irreducible. Some results can be strengthened for irreducible polynomials. However, we do not want to look at irreducible separately because the GLC forces us to deal with reducible polynomials anyways due to the emergence of the necessary cofactor we require in order to be able to overcome dimensionality restrictions of the MSLMPR (monic symmetric linear matrix polynomial representation) we search for.

\begin{remark}[All polynomials in sight]
Even if we chose to set the initial polynomial $p$ to be irreducible, there is a forced cofactor $q$ implied in the GLC that could be (highly) reducible itself and, anyway, the product obtained $qp$ being representable as a determinant of a MSLMP is going to be reducible itself being the product of the original polynomial and the sought cofactor.
\end{remark}

This focus on real roots has a direct consequence on our notions of \textit{distance} between polynomials. But this requires a treatment by itself that forces us to leave this preliminar section

\section{Measuring closeness of polynomials}\label{measuring}

We made an effort in Proposition \ref{wlogcom} to stay within the affine setting all the time. This effort could have taken a different form if we acted in a different way on this section looking through the projective lenses. However, we prefer to stay within the affine setting so that the next measures make sense in a more straightforward and natural way.

Notice that we avoided using the term \textit{distance} in the title above. This is so because we are not going to require these \textit{measures} to be distances in the usual mathematical sense of that word. We only require that they behave well locally in a way that allows us to build a local sense of closeness \textit{along sufficiently many lines}. That is, we just need a tool that let us recover some notion of closeness from the Euclidean topology in the set of polynomial when we look at it by their roots and not by their coefficients. This notion of closeness will turn out to be the same by simple continuity arguments related to the roots.

\begin{definicion}[Naive polynomial distance]\label{naivepoldis}
Let $p,q\in\mathbb{R}[\mathbf{x}]$ be polynomials and denote $\coeff(m,p)$ the coefficient of the monomial $m$ in $p$. Then, for $u\geq1$, we define the \textit{$u$ coefficient distance} between $p$ and $q$ as $$d_{c}(p,q)=\left(\sum_{m\in M}|\coeff(m,p)-\coeff(m,q)|^{u}\right)^{1/u}$$ where $M$ is the set of all monic monomials in the variables $\mathbf{x}$. 
\end{definicion}

In this way, the set of polynomials behaves like the vector space $c_{00}$ of eventually-zero real valued sequences equipped with the metric given by the $u$-norm. We can \textit{twist} things a bit in oder to build a metric that allows us to visualize things better.

\begin{definicion}[Twisted polynomial distance]\label{twistedpoldis}
Let $p,q\in\mathbb{R}[\mathbf{x}]$ be RZ polynomials and denote $\rt_{m}(a,p)$ the $m$-th root in the set of roots of the polynomial $p$ along the semiline positively spanned by $a$, that is, along the semiline $\mathbb{R}_{\geq 0}a=\{ta\mid t\in\mathbb{R}_{\geq 0}\}$ where the order goes from closest to farthest to the origin $\mathbf{0}.$ Then, for $u\geq1$, we define the \textit{$u$ root distance} between $p$ and $q$ as $$d_{r}(p,q)=\int_{a\in\mathbb{S}^{n-1}}\sum_{i}|\rt_{i}(a,p)-\rt_{i}(a,q)|^{u},$$ where the sum goes until $i$ provokes one of the sets of roots to run out of elements.
\end{definicion}

If our two polynomials $p$ and $q$ have compact real zero set (which is the generic case we decide to deal with here) then this measure considers the corresponding ovaloids of zeros $P=\{P_{1},\dots,P_{\deg(p)}\}$ and $Q=\{Q_{1},\dots,Q_{\deg(q)}\}$ and, until $d=\min\{\deg(p),\deg(q)\}$, it computes basically the total of the volumes of the differences $\overline{Q}_{i}\smallsetminus\overline{P}_{i}$, where $\overline{O}_{i}$ for the ovaloid $O_{i}$ means the connected component of the origin in the set $\mathbb{R}^{n}\smallsetminus O_{i}$.

\begin{remark}\label{heltonmeasure}
Looking at the different ovaloids of the RZ polynomials $p$ and $q$ in the sense described in \cite{helton2007linear}, this measure of distance is actually a measurement of how close are all the zeros of $p$ from a zero of $q$ along each line. This extends naturally to a measure of closeness in the whole space.
\end{remark} 

The measure moreover is able to compute this sense of closeness even when the two polynomials have different degree and therefore vary in the number of roots along each line. This is actually a strength of this way of measuring things contrary to the naive one because it measures closeness not towards the original polynomials, which, for us, might be actually unimportant, but to possible multiples of these polynomials. That is, this measure is a measure of closeness towards the possible multiples that interest us and that is able to disregard multiples that we do not in principle care about in order to detect our required notion of closeness. This will be important because of the forced appearance of cofactors in our argument that we do not need to control completely beyond knowing that they are there for algebraic reasons and that they are \textit{far enough} from our rigidly convex set of interest.

\begin{remark}[A better measure of closeness]
By disregarding, non-matching roots we measure closeness of a polynomial $p$ with a close factor of $q$ instead of looking at the whole $q$. This measure therefore gets automatically rid of disturbances in the notion of closeness introduced by far away factors within $q$. This allow us to concentrate on the topological behaviour detaching it from the algebraic part that will still play a role in the process of having a determinantal representation.
\end{remark}

Observe that there is an inconvenience in the fact that we count each line two times in the Definition \ref{twistedpoldis}. This can be easily solved taking just half a sphere. Denote $\mathbb{S}^{n-1}_{+}:=\{a\in\mathbb{S}^{n-1}\mid a=(\mathbf{0},a_{i},\dots a_{n}) \mbox{ with } a_{i}>0\}\subseteq\mathbb{S}^{n-1}\subseteq\mathbb{R}^{n}.$ Thus we only take half the lines vectors in $\mathbb{S}^{n-1}$ and span each line just once. Thus we can define instead the measure as follows.

\begin{definicion}[Half twisted polynomial distance]\label{halftwis}
  Let $p,q\in\mathbb{R}[\mathbf{x}]$ be RZ polynomials and denote $\rt_{m}(a,p)$ the $m$-th root in the set of roots of the polynomial $p$ along the line spanned by $a$, that is, along the line $\mathbb{R}a=\{ta\mid t\in\mathbb{R}\}$ where the order is given by the order of the corresponding values of $t$ for each point so that the order depends on the parametrization. Then, for $u\geq1$, we define the \textit{$u$ half root distance} between $p$ and $q$ as $$d_{h}(p,q)=\int_{a\in\mathbb{S}^{n-1}_{+}}\sum_{i}|\rt_{i}(a,p)-\rt_{i}(a,q)|^{u},$$ where the sum goes until $i$ provokes one of the sets of roots to run out of elements.
\end{definicion}

However, this does not actually matter in the end, because this will only change in the end coefficients of closeness and we are not looking at these. In fact, in the practice, by trivial continuity arguments, we only need to check this closeness along finitely many close lines at a time.

\begin{proposicion}[Measuring finitely]
Let $p$ and $q$ be smooth RZ polynomials with compact zero set. For all $\epsilon>0$ there exist a choice of lines through the origin $M$ and a partition $\{S_{l}\mid l\in M\}$ of $S^{n-1}$ into measurable sets with $l\in S_{l}$ such that $|d_{r}(p,q)-\sum_{l\in M}\vol(S_{l})\sum_{i} |\rt_{i}(a,p)-\rt_{i}(a,q)|^{u}|<\epsilon,$ where for a measurable subset $S\subseteq S^{n-1}$ the quantity $\vol(S)$ is its usual Lebesgue volume inside this space.
\end{proposicion}

\begin{proof}
Evident by continuity.
\end{proof}

Thus if we want to ensure $p$ being close enough to a possible factor of $q$, we can do this looking at what happens in a finite number of lines through the origin. We choose these distances because they are the most natural ones. It would have also being natural to consider just the distances obtained after substituting $\sum$ by $\max$ and $\int$ by $\sup$. Notice that, in any of the cases, the newly introduced measures here tell us more algebraic information about the decomposition of the polynomial than the naive measure by a distance. And this happens precisely because $d_{r}$ is not a distance in the usual terms. This is what the following theorem tells us.

\begin{teorema}[Distance and multiples for RZ polynomials]\label{distmult}
Let $p,q\in\mathbb{R}[\mathbf{x}]$ be RZ polynomials. Then \begin{enumerate}\item $d_{c}(p,q)=0$ iff $p=q$.\\\item $d_{r}(p,q)=0$ iff $p|q$ or $q|p$.\\\item $d_{h}(p,q)=0$ iff $p|q$ or $q|p$.\end{enumerate}
\end{teorema}

\begin{proof}
The proof is again immediate by continuity and Theorem \ref{comonicreal}.
\end{proof}

Observe that Definitions \ref{twistedpoldis} and \ref{halftwis} involve RZ polynomials. This is so because otherwise we cannot order the roots. But also because non-RZ polynomials will have, in terms of roots, an insurmountable measure of distance in our sense here. The following example shows us why we should expect problems when the polynomials are not RZ.

\begin{ejemplo}[The TV-screen]
The polynomial $p=1-x_{1}^{4}-x_{2}^{4}$ is the classical example of a polynomial having a convex real zero set forming an ovaloid around the origin but not being a RZ polynomial. This is a counterexamples for the tentative of claiming that only RZ polynomials have this innermost ovaloid convex. We would also like to be able to measure distance for this kind of polynomials because they play the role of representing an obstacle to our approach that only can be discarded through this distance.
\end{ejemplo}

We can solve this problem thinking ahead. In particular, we should think about a multivariate polynomial as entirely given in terms of its behaviour around its lines through the origin no matter what the roots are. The difference with the approach above is that we now take our line complex and study what happens in each line without fearing complex roots. The measure that we produce in response will in fact be the one we pursue here. For this, we will first improve Propostion \ref{comonicreal}.

\begin{proposicion}\label{comoniccomplex}
If two comonic polynomials $p,q\in\mathbb{R}[\mathbf{x}]$ have coinciding roots along each line through the origin, then $p=q$.
\end{proposicion}

\begin{proof}
Now we do not only look at the real roots. We look at all the roots through each line. These roots will again completely determine each line restriction along all real lines. This information completely determines the behaviour of the polynomial in $\mathbb{R}^{n}$ and this behaviour in fact completely determines the polynomial formally (and thus in all of $\mathbb{C}^{n}$).
\end{proof}

Now we want to use the distance in each line to define our global distance. However, now we do not have a way to order the roots, so we have to resort directly to the metric given by the absolute value in the corresponding complex line in order to establish the behaviour in each line.

\begin{definicion}[Radial Hausdorff-Fréchet distance]\label{rhfd}
Let $p,q\in\mathbb{R}[\mathbf{x}]$ be polynomials and denote $\rt(a,p)$ the set of roots of the polynomial $p$ along the complex line spanned by $a\in\mathbb{R}^{n}$, that is, $\mathbb{C}a$. Choose the one polynomial with the lowest degree, say $p$. Then, for $e$ a distance in the complex plane, we define the \textit{$e$-radial distance} between $p$ and $q$ as $$d_{t}(p,q)=\int_{a\in\mathbb{S}^{n-1}_{+}}\sum_{b\in\rt(a,p)}e(b,\rt(a,q)).$$
\end{definicion}

The name comes from three main origins that inspire this measure. More can be read in \cite{pompeiu1905continuite}, \cite{frechet1906quelques} and \cite{hausdorff1978grundzuge}. The radial part is clear in the sense that it measures the distance between these polynomials along the radii provided by each line through the origin in $\mathbb{R}^{n}$.

\begin{teorema}[Distance and multiples for all polynomials]\label{distmultall}
Let $p,q\in\mathbb{R}[\mathbf{x}]$ be polynomials. Then $d_{t}(p,q)=0$ iff $p|q$ or $q|p$.
\end{teorema}

\begin{proof}
The proof is immediate as it forces the smallest degree polynomial to have overlapping zero set with the biggest degree one along real lines. Thus, this determines the behaviour of the polynomials along (lines through the origin in) $\mathbb{R}^{n}$ and, again by Theorem \ref{comonicreal}, formally the polynomial itself, forcing therefore one polynomial to divide the other.
\end{proof}

A measure does not define a topology in general, but we can try to work around this. For this we use the following statregy based in the fact that we just need to measure closeness as our topology is already defined (the Euclidean one) and therefore we do not need to be able to define it completely though our measures of closeness which are not distances.

\begin{remark}[Jumping above topological requirements]\label{jum}
The topology of the space of polynomials and of the space where it takes zeros is already defined and determined: in both cases here the corresponding Euclidean topology. For this reason, we do not need a tool able to reconstruct the whole information about such topology, we just need a measure of how far one polynomial is of being a factor of another.
\end{remark}

Thus, the topology we use in the set of polynomials of some fixed degree $d$ by this measure is clearly the same one induced by the usual norm defined in \ref{naivepoldis}. In this topology our measure behaves as we expect because the roots depend continuously on both the (angle of the) radius taken and the coefficients of the original polynomial. The only difference is that this new measure identifies polynomials of higher degree with different polynomials in degree $d$ and thereofore does not extend to a metric in the whole space of polynomials. Now it is obvious to see that a non-RZ polynomial, like e.g., the TV screen will have strictly positive distance with any RZ polynomial.

\begin{ejemplo}[Positive distance with non RZ polynomials]\label{TVscreen}
The metric above separates RZ polynomials from non-RZ polynomials. The TV screen $p$, with its complex roots, lies always far from any RZ polynomial because there is some $\epsilon>0$ such that $d_{r}(p,q)>\epsilon$ not matter which RZ polynomial $q$ we choose.
\end{ejemplo}

And it is also clear in this way that (in the Euclidean topology of polynomials) the open sets around one RZ polynomial in the space of RZ polynomials can be easily visualized in the space of zeros as a cloud around the real zeros of such polynomial. Such cloud represents how can the zeros change when we take any small enough path in the Euclidean space of polynomials. These movements around zero are obviously \textit{rigid} in the sense that they have to have \textit{polynomial form} while the movements in the space of polynomials are not so because they are free movements in that space. Indeed this will be something important to have in mind during our explorations.

\begin{remark}[Rigidity in different representations of the paths]
A path in the space of RZ polynomials from $p$ to $q$ is just a (suitable) homotopy in the space of RZ polynomial, that is, some continuos map $H\colon[0,1]\to\mathbb{R}[\mathbf{x}]_{\mbox{RZ}}$ with $H(0)=p$ and $H(1)=q$. For our proof we will ask a bit more from this homotopy, like it happens to take smooth polynomials in the open interval $(0,1)$ (so that we can ensure that it respects $\rcs(p)$). But that is all. The dependence on $t$ does not need to be polynomial. However, when we look at what this means in the space where we visualize the zeros of the polynomial, the transformation performed by the homotopy $H$ when $t$ varies in $[0,1]$ follows a rigid path in $\mathbb{R}^{n}$ meaning that the small perturbations that we visualize along the whole movie represented by $H$ along $t\in[0,1]$ happen to deform and modify our ovaloids jumping along polynomials and thus not taking place in just a continuos way but also in a rigid (polynomial) way. This is what we have to remember along these paths so we do not confuse what happens in one space (the Euclidean space of polynomials) and in the other (the rigid space of the deformations of its roots).
\end{remark}

Thus we see how the constrained rigid sets in one space correspond with open sets in another space. We will use this as a tool to jump from the rigid space (deformations of real zero ovaloids in $\mathbb{R}^{n}$) to the unconstrained one (paths in the space $\mathbb{R}[\mathbf{x}]_{\mbox{RZ}}$) in order to obtain the polynomials we seek. In order to do this, we make use of determinantal representations that we already know well.

\section{Determinantal machinery}

In this section we explain why the determinant is expected to behave so universally. This will help us understand better the proof of the covering of opens that we will present in the next section.

When the initial matrix is not required to be the identity but any other signature matrix we have symmetric determinantal representations for all polynomials. This result tells us how powerful is (the ability of) the determinantal form to represent all polynomials through determinants of SLMP. There are several proofs available for this theorem using different techniques, see \cite{helton2006noncommutative}
\cite{quarez2012symmetric} or \cite{stefan2021short}.

\begin{teorema}[Helton-Vinnikov-MacCollough determinantal representation]\label{hvm}
Let $p\in\mathbb{R}[\mathbf{x}]$ be a polynomial of degree $d$ in $n$ variables with coefficients in $\mathbb{R}$ and such that $p(\mathbf{0})\neq0$. Then, for $N=2\binom{n+\lfloor\frac{d}{2}\rfloor}{n}$, there exist matrices $J,A_{1},\dots,A_{n}\in\Sym_{N}(\mathbb{R})$ with $J$ a signature matrix such that $$p(x)=p(\mathbf{0})\det(J)\det(J+x_{1}A_{1}+\cdots+x_{n}A_{n}).$$
\end{teorema}

This theorem does not provide the kind of representations that we want. In fact, notice that it works for \textit{all polynomials} and not just the RZ ones. The signature matrix $J$ will in general have (diagonal) entries with value $-1$ while we want to be able to have a symmetric representation with an identity in its place instead (monic).

\begin{remark}
Taking a signature matrix $J$ instead of an identity matrix opens the possibility of representing all polynomials as determinants. This is a doubly edged sword for us. On the one hand, this says that this framework cannot distinguish between RZ and non-RZ polynomials. On the other hand, it tells us that the determinantal form is powerful enough to carry the representation of all the polynomials. Now we just have to be able to establish this same power for \textit{nice} (RZ respecting original RCS) multiples of RZ polynomials using just the identity matrix as our initial matrix. 
\end{remark}

For this reason, we do not use this theorem for the representation it provides, which is of no help for us. We use it for what it tells us about the power (universality) of the determinants of SLMPs of that size $N$ to represent the \textbf{whole} \textit{metric (vector)} space of all polynomials of degree $d$ in $n$ variables. In particular, this representation tells us that all such polynomials \textit{admit} this form and therefore, looking at it the other way around, that varying all the entries in the SSMLPs (Signature Symmetric MLPs) given of this size $N$ and applying the determinant has to end up giving as a result the whole space of polynomials of degree $d$ in $n$ variables. Some form of generic unicity in the representation of product of linear forms will be central for the following task. We develop it as its own section due to importance of the next result.

\section{Opens around determinants}

The section above showed that symmetric determinants are universal in the sense that they manage to represent the whole space of polynomials if we allow the initial matrix to be a signature matrix. However, this is bad for us in certain sense: we want to use symmetric determinants to represent only RZ polynomials. This tells us that a further restriction has to be made. The restriction we take is to fix the initial matrix to be a PD matrix and, wlog, we can take this matrix to be the identity matrix. Taking this matrix will avoid that certain cancellations happen and, in consequence and in accordance with Theorem \ref{dime}, we will end up producing multiples of our polynomials. However, we will eventually ensure that these multiple are nice in the sense that they respect the original RCS. The theorem we wanto prove next is the folllowing.

\begin{teorema}
Let $L=\prod_{i}^{s} l_{i}$ be a product of linear forms with $l_{i}(\mathbf{0})=1$ for some size $s$ big enough depending on $d$ and $n$. Then any polynomial $p$ of degree $d$ close enough to $L$ admits a cofactor $q$ with $q(\mathbf{0})=1$ and a determinantal representation that is a perturbation of $\det(\diag(l_{i}))$. That is, $qp=\det(I+\sum x_{i}A_{i})$ with $\dist(I+\sum x_{i}A_{i},\diag(l_{i}))\leq\epsilon$ for some distance $\dist$ between matrix polynomials.
  \end{teorema}

\begin{proof}
One can see $L$ as a solution of the equation $\det(I+\sum x_{i}A_{i})$ for the entries of the symmetric matrices $A_{i}$ as variables. Expanding $\det(I+\sum x_{i}A_{i})$ produces a polynomial in the variables $x_{i}$ and $a_{ij}$. We can now use an ansatz to force the form \begin{gather*}\det(I+\sum x_{i}A_{i})=pq\end{gather*} for $p$ of degree $d$ and $q$ of degree $s-d$. This can be done because of continuity of the operations involved in taking the determinant and decomposing into factors of degree $d$ and $s-d$ whenever this is possible. Letting thus $s$ grow will allow for more freedom in the coefficients on the set of variables $\mathbf{x}$ of $p$ paying the only price of having $q$ of higher degree, which is a price that is small for us. With $s$ big enough, as $p$ has fixed degree, we have introduced enough freedom in these coefficients to allow them to change after certain compensations are made in $q$. Thus, as $p$ being close enough to $L$ means that we only need to make very small changes in the cofficients of $L$, we can gaurantee that for $s$ big enough and $p$ close enough to $L$, the polynomial $p$ admits a cofactor $q$ such that $qp$ with a representations given by a small perturbation $I+\sum x_{i}A_{i}$ of $\diag(l_{i})$. This finishes the proof.
\end{proof}

An equivalent form of writting this theorem is even more revealing about the power of what we have just proved.

\begin{teorema}\label{keyth}
  Let $p$ be a degree $d$ RZ polynomial sufficiently close to a product of linear forms $L=\prod_{i}^{s} l_{i}$ with $l_{i}(\mathbf{0})=1$ for some size $s$ big enough. Then $p$ admits a cofactor $q$ with $q(\mathbf{0})=1$ and a determinantal representation that is a perturbation of $\det(\diag(l_{i}))$. That is $qp=\det(I+\sum x_{i}A_{i})$ with $\dist(I+\sum x_{i}A_{i},\diag(l_{i}))\leq\epsilon$ for some distance $\dist$ between matrix polynomials.
\end{teorema}

All this tells as that, when we have a polynomial $L$ with a determinantal representation of size $s$ big enough, the perturbations of such representation cover the space of all degree $d$ polynomials around it which are sufficiently close. And it turns out that, because of compactness, convexity and smoothness, that is all we need to succeed.

\section{Finitely covering a real zero set through compactness}

The idea of performing infinitely close approximations to represent similar objects has appeared before in the literature (see \cite{netzer2015smooth}), but on my consideration this line of attack was not exploited enough and in suficcient depth. In particular, the key point is to reach to use just a finite number of steps after which a deformation just falls within the reach of the possibilities of the expressions we need using the power of the theorems introduced on the sections above. This is what we do here.

\begin{remark}[Form infinite to finite]
We will use compactness to allow for a representation of finite size that can easily be deformed into the polynomial we want to represent using Theorem \ref{keyth}. 
\end{remark}

When the product of linear forms are not close enough to our polynomial objective $p$, we will have to build something that is close enough and that can be deformed into it. For this, we use families of tangent hyperspaces. Remember that, for simplicity and without loss of generality, we assume $p$ to be RZ, comonic, generic and smooth and with compact real zero set.

\begin{construccion}[Tangent product]\label{cons}
Let $\binom{V}{N}$ be the collection of subsets of cardinal $N$ of points of the real zero set $V$ of $p$. For each such collection $U\in\binom{V}{N}$ we build a product of linear forms $p_{U}:=\prod_{u\in U}l_{u}$ with $l_{u}$ tangent to $p$ at the point $u\in V$ and $l_{u}(\mathbf{0})=1$. Remember that $p$ is generic (and therefore smooth).
\end{construccion}

The open expressible as a perturbation around $p_{U}$ gets in principle very close to the real zero set $V$ but it does not have to match it. We need to ensure that this happens providing enough freedom of movement through increases of its degree.

\begin{remark}[Rigidity is a problem for closeness]
As the movements as seeing in the set $\mathbb{R}^{n}$ where the real zeros of $p$ are is a requirement, it is clear that a perturbation can lie very close without it being allowed to actually reach a multiple of the polynomial itself. This is so because the open bands reachable around the ovaloids of the approximating product of linear form $p_{U}$ might get very thin in points that are far from the zero set $V$ of $p$. This fact implies that the rigid movements do not have enough room to go in the direction of $p$ following the lines through the origin corresponding to these thin parts of the band around the ovaloids of $p_{U}$.
\end{remark}

Thus the products do not in principle cover a whole open around the ovaloids of the real zero set $V$ of our initial polynomial $p$. We need to increase the degree of our approximating product of linear forms several times having in mind that the opens covered could change when we change the size of the matrices. However, this change will always go in our favor when the size of the matrix representation increases.

\begin{teorema}[Matrix growth increase open bands]\label{openbands}
Let $S=I_{s}+\sum_{i=1}^{n}x_{i}A_{i}$ and $T=I_{t}+\sum_{i=1}^{n}x_{i}A'_{i}$ be MSLMP of sizes $s$ and $t$ respectively. Suppose that $S$ can be deformed so that its determinant can cover RZ polynomials in an open band $U_{S}$ around each ovaloid of $\det(S)$ and the same for $U_{T}$. Then $S\oplus T$ is a MSLMP of sizes $s+t$ which can be deformed to cover (at least) all RZ polynomials in an open band $U_{S}\cup U_{T}$ of the ovaloid of the real zero set of $\det(S)\det(T)$.
\end{teorema}

\begin{proof}
Obvious performing the deformation in $T$ and in $S$ independently. In fact, $S\oplus T$ has extra room (out of the block diagonal) for deformation in a smooth way that still gets you closer to the real zero set $V$ of $p$.
\end{proof}

We obviously can do this sequentially. Thus, we will tie together (diagonally) several of these products of linear form in its trivial matrix representation presented in Construction \ref{cons} until we cover completely all the ovaloids of the real zero set $V$ of our original polynomial $p$.

\begin{construccion}[Tangent cover]\label{tangcov}
The opens of allowed determinantal perturbations reach of all the $p_{U}$ with $U\in\binom{V}{N}$ covers the whole real zero set $V$ with rigid opens around ovaloids. But as $V$ is compact (because $p$ is generic), to cover it we just need to use a finite subset $F\subseteq\binom{V}{N}$. Thus the product of linear forms $P=\prod_{U\in F}p_{U}$ has a determinantal representation (trivial) that can be rigidly deformed into a determinantal representation of a multiple of $p$ because the opens of the rigid deformations allowed for this product cover the whole ovaloids by compactness.
\end{construccion}

The first thing to note in this construction is that we have to use the deformation for each collection $U$ of $N$ hyperplanes (tangent hyperplanes to points in the ovaloids of the real zero set) because changing the size of the representation could change their expressivity (size of the opens near points in the ovaloids). For this reason, we keep this size constant and just glue several size $N$ representation by taking the long diagonal matrix corresponding to their concatenation (direct sum). This process is going to produce a huge matrix, which is something that we would expect for the RCSs as we said above. In principle, we do not even have control of the size obtained, as this depends on a measurability of compactness under the determinant map. The form of the product of linear forms $P$ obtained is also helpful because it will respect the original RCS of $p$. This happens because we took tangent hyperplanes and deformed them along paths of smooth polynomials, which guarantee that no intersections take place and therefore that the RCS of $p$ is respected at the end of the path.

\begin{remark}[Expected form of things]
Thus at the end we will get a small perturbation of a huge block diagonal MSLMP that will produce a MSLMP whose determinant is a \textit{nice} factor of the original polynomial obtained through bendings of hyperplanes tangent to the original ovaloids along rigid smooth polynomial paths (as seen in the set $\mathbb{R}^{n}$ where the zeros live).
\end{remark}

We will also need some notions of the relation between smoothness and convexity while performing a deformation. These deformations will allow us to traverse the open set around our original polynomial in a satisfactory way.

\section{Convexity, smoothness and the ability to maintain the RCSs untocuhed along paths}

This section is a futher consequence of Theorems \ref{nuijth} about the topological simplicity of the space of RZ polynomials and Theorem \ref{hvm} for the universality of determinantal representations. We observe here in detail how smooth deformations of tangent hyperplanes conducting to producing a multiple of the polynomial $p$ as described in the section above are forced to respect the rigidly convex set $\rcs(p)$ of $p$ by its nature. Hence, in this way the GLC comes naturally from the nature of the deformations performed. We remind a result that says that we can move the origin within the RCS and still get a RZ polynomial.

\begin{teorema}[RZ-ness is conserved under small enough perturbations of the origin]
Let $p\in\mathbb{R}[\mathbf{x}]$ be a RZ polynomial and $\mathbf{x}_{0}\in\rcs(p)$ with $p(\mathbf{x}_{0})\neq0$. Then the polynomial $q(\mathbf{x}):=p(\mathbf{x}-\mathbf{x}_{0})\in\mathbb{R}[\mathbf{x}]$ is RZ and its RCS is just a translation by $\mathbf{x}_{0}$ of the original $\rcs(p)$, i.e., $\rcs(p)=\rcs(q)-\mathbf{x}_{0}.$
\end{teorema}

\begin{proof}
This is well known.
\end{proof}

This means in particular that we can change the origin arbitrarily within the RCS when we want to check RZ-ness of a polynomial. For the next result, define in general $\rcs(p)$ the connected component of the origin in $\mathbb{R}^{n}\smallsetminus V$ for \textit{any} polynomial $p$.

\begin{corolario}\label{coroRZ}
The following are equivalent for a polynomial
$p\in\mathbb{R}[\mathbf{x}]$
\begin{enumerate}
  \item $p$ is a RZ polynomial
  \item for a point $\mathbf{x}_{0}\in\rcs(p)$ with $p(\mathbf{x}_{0})\neq 0$ the restrictions to all the lines $l$ through $\mathbf{x}_{0}$ produce real-rooted polynomials, i.e., if for all $v\in\mathbb{R}^{n}$ the polynomials $p(\mathbf{x}_{0}+vt)\in\mathbb{R}[t]$ is real-rooted.
  \item for all the points $\mathbf{x}_{0}\in\rcs(p)$ with $p(\mathbf{x}_{0})\neq 0$ the restrictions to all the lines $l$ through $\mathbf{x}_{0}$ produce real-rooted polynomials, i.e., if for all $v\in\mathbb{R}^{n}$ the polynomials $p(\mathbf{x}_{0}+vt)\in\mathbb{R}[t]$ is real-rooted.
\end{enumerate}
\end{corolario}

Now we can use this corollary to prove the next fundamental theorem.

\begin{teorema}[Convexity and smooth deformations guarantee respectfulness]\label{respect}
Let $p\in\mathbb{R}[\mathbf{x}]$ be a generic (therefore smooth and with compact set of real zeros) RZ polynomial and $\prod_{i=1}^{s}l_{i}\in\mathbb{R}[\mathbf{x}]$ a product of tangent linear forms to the real zero set $V$ of $p$ with $l_{i}(\mathbf{0})=1$ such that the MSLMPR $\det(\diag(l_{i}))$ admits a deformation  $H\colon[0,1]\to\mathbb{R}[\mathbf{x}]$ starting at it and ending up in a RZ multiple $qp$ of $p$ along a path of smooth polynomials in $\mathbb{R}[\mathbf{x}]_{RZ}$. Then $\rcs(p)\subseteq\rcs(H(t))$ for all $t\in[0,1]$. This implies in particular $\rcs(p)=\rcs(qp)$.
\end{teorema}

\begin{proof}
First, convexity of the RCS of $p$ implies that the hyperplanes $l_{i}$ tangent to its border (innermost ovaloid of zeros of $p$) do not cross through $\rcs(p)$ in a smooth perturbation and therefore small enough smooth perturbations of these hyperplanes stay outside of $\rcs(p)$. Moreover, hyperplanes $l_{i}$ tangent to other zeros not in the innermost ovaloid cannot cross the set $\rcs(p)$ as this would break smoothness at some point in the deformation. The smooth perturbations of the hyperplanes along a path of smooth RZ polynomials gaurantees that the multiple zero sets of the product (that is the points of intersection of the  hyperplanes) nearest to $\rcs(p)$ of the initial product of tangent hyperplanes break in the direction approaching smoothly towards $\rcs(p)$. Otherwise, in view of Corollary \ref{coroRZ}, for some $t>0$, the polynomial $H(t)$ would not be RZ as the rupture in a tranversal direction would allow for a line with origin in a point within $\rcs(p)$ passing by this rupture to have a drop in the zeros resulting in a non real-rooted univariate restriction along this line. Thus, after $t=0$, by smoothness, no further intersections are are allowed (after $t=0$ all the ovaloids are defined by the branches they deform already) and therefore the border of the RCS of $H(1)=qp$ has to cover the border of $\rcs(p)$ without further branches of zeros entering inside $\rcs(H(t))$ for $t>0$. This implies that at the end of the deformation, we must have $\rcs(p)=\rcs(qp)$ as not further branches could otherwise get close enough to the zeros corresponding to the border of $\rcs(p)$ because this would produce intersections (and therefore break smoothness of the corresponding polynomials) at points $t\in(0,1)$.
\end{proof}

Now we just have to combine all the results we saw in order to produce a proof, which is basically already done at this point. We do this combination explicitly in the next section for clarity.

\comm{
\section{Transducting compactness}

going from the compactness of the zero set from the compactness of the zero set to the compactness of the sets around the point representing the polynomial.
}

\section{Combining results and observation into a proof}


We can write down the main statement. The results above combine to prove the following.

\begin{teorema}\label{glc}
  The generalized Lax conjecture is true.
\end{teorema}

\begin{proof}
We start with a RZ degree $d$ polynomial $p$ in $n$ variables. Wlog, we suppose that $p$ is generic, smooth and with compact real zero set. We use Proposition \ref{comoniccomplex} to focus our attention to the real zero set $V$ of $p$. As $p$ is RZ, if we manage to form another (RZ) polynomial that has the same roots as $p$ along each real line through the origin, we can ensure that it coincides with $p$. We will form this polynomial as a factor of a bigger polynomial. We choose a radial distance $\dist$ around the ovaloids of $p$. Perturbing $p$ within a small enough neighbourhood of their ovaloids with respect to this distance corresponds to small perturbations of its coefficients. Choose a family $\{L_{1},\dots L_{k}\}$ of products of linear forms $l$ with $l(\mathbf{0})=1$ such that the union of the opens guaranteed by Theorem \ref{keyth} cover the whole RZ set of $p$ (all its ovaloids). Choose moreover the $L_{i}$ with degree high enough to allow for deformations covering opens around polynomials of degree $d$. By compactness of the RCS of $p$ (by assumption), such $k$ exists and it is finite. We can choose these products of linear forms $L_{i}$ following Construction \ref{cons}. Now, using Theorem \ref{openbands} we can glue together these $L_{i}$ forming a MSLMP allowing a deformation producing a multiple $qp$ of $p$. By Theorem \ref{simpli}, as this glueing allows the expression of a whole open around $p$, this deformation can happen in through smooth RZ polynomials. Now Theorem \ref{respect} guarantees that, along such smooth deformation, $\rcs(p)$ is respected and therefore $\rcs(qp)=\rcs(p)$ and $qp$ admits a MSLMPDR (monic symmetric linear matrix polynomial determinantal representation). This finishes the proof.
\end{proof}

The theorem above settles our research about approximations of RCS through adequate linear forms. In fact, the proof tells us even more as it tells us that we can reach the desired representation through deformations of linear forms sufficiently close to the original polynomial without requiring strange point-wise representations that pop-up out of some algebraic structure. The proof above tells us that this is a topological property of approximations of RCSs by linear forms and not necessarily an algebraic property. The algebraicity comes when one wants to consider notions of minimal representations or representations with some additional structure. This is a nice question by itself that goes into the computational or effective setting. However, now that we know that topologically such representation exists, it is the correct time to ask the following (more algebraic) question.

\begin{cuestion}
Given a RZ polynomial $p$, does there exists an algorithm giving a MSDR representation of $p$?
\end{cuestion}

\section{Objections and journey}

Approaches like the one followed here to produce a prove of the GLC might have been seen as problematic before because of a number of reasons. In particular, one encounters these objections while thinking about this problem and they can deter some attempts to build a proof as que did here. But these objections actually push in a natural way towards the correct solution. We comment how this happened here. 

The first objection was commented in the first section and has to do with the dimensionality obstruction. However, this objection just tells us that we need a strong mechanism able to produce \textit{astronomically} big cofactors in a \textit{natural} way. Our natural way here was to look at the topology due to the uncontrollable nature of the algebraic requirements of these huge objects. And, in this way, this obstruction pushed us towards looking for the topological proof presented here.

\begin{remark}[Dimensionality explosion calls for topology]
Theorem \ref{dime} already points out to the fact that controlling the cofactor agebraically can be problematic. Theorem \ref{uncon} deepens into this problem even for \textit{just} approximations, not to talk about actual representaions. However, Theorem \ref{simpli} tells us that the topology of the space of RZ polynomials is easy to deal with. The combination of these facts points immediately to the topologicla approach we took there.
\end{remark}

Theorem \ref{uncon} introduces naturally the concept of approximations into our setting and inquirirs. We will quickly evolve from mere approximation to the more helpful concept of deformation. First, however we have to formalize these concepts. This is what pushes us to consider \textit{measures of distances} around ovaloids.

\begin{remark}[Measuring approximations]
In order to understand the obstruction posed by Theorem \ref{uncon}, we need to formalize the concepts of approximations introducing a notion of distance. This is done in Definition \ref{defapr} and Remark \ref{defapr}. This initial assessment pushes the developments in Section \ref{measuring}.
\end{remark}

As we decide to immerse ourselves in the topology, we observe that we have to distinguish between two topolgies that will be deeply related: the topology of the set of real zeros and the topology of the set of RZ polynomials. This is what we do in Remarks \ref{twoviews} and \ref{compactnessviews}. These first view into the topological nuances of our problem allow us to reduce the setting without loss of generality through Proposition \ref{wlogcom}. 

\begin{remark}
Thanks to a few topological observations we can readily reduce the setting to look only at polynomials with very nice real zero sets. In particular, we focu our attention in smooth, comonic, generic RZ polynomials with compact real zero set. And we prove that we can do this without losing generality but gaining much more clarity. This reduction will be crucial to develop the ultimate proof of the GLC we provide here.
\end{remark}

We establish what information completely determines a polynomial. This leads us firs to Theorem \ref{comonicreal}, which will evolve into Theorem \ref{comoniccomplex}. This information however will give as a key point about how should we think about our approxiamtion to the zero set of $p$ if we want to be successful in our quest.

\begin{remark}[Avoiding the TV screen objection]
Along Section \ref{measuring}, whenever we measure closeness of polynomials, we take care of measuring closeness with respect to all the ovaloids (connected components) composing the real zero set. We do this to ensure that our notion of closeness as measured by these functions has a meaning in terms of factors as shown in Theorems \ref{distmult} and \ref{distmultall}. In particular, getting closer to \textit{all} ovaloids (and not just some of them) is what allow us to use Theorem \ref{comonicreal} to establish Theorems \ref{distmult} and \ref{distmultall} because thanks to looking at all ovaloids we can ensure that, on each line through the origin, the roots coincide, allowing us to effectively apply the theorem. Thus, as shown by Example \ref{TVscreen}, the approach we take here cannot work for RZ polynomials because, in terms of these measures, they are always too far from our approximations using products of linear forms, no matter how much we increase the degree of the product. This is an important point because one might easily fall in the trap of trying to approximate just the RCS of $p$ using products of linear forms. If one do this, everything else in our proof would seem to work in the case of the TV screen polynomial because its real zsro set is smooth and convex. However, as it is not a RZ polynomial, this would tell us that there is something wrong with approach. What is wrong is to forget to measure the distance to non-real roots over real lines not the approach itself. We took so much care about this here because this example might have made people skeptical of this approach before we introduced the complete measures in Section \ref{measuring}.
\end{remark}

In regard to these measures, the compactness assumption is also of great help. In particular, as noticed before Remark \ref{heltonmeasure}, we are measuring volumes of differences of measurable sets with finite measure.

\begin{remark}[Compactness and finiteness]
This finiteness implies that, for RZ polynomials, these measures always give us a finite number that we can reduce increasing our accuracy through additional linear forms. This feature is something that, as we see ultimately, is key in our approach.
\end{remark}

Additionally, these measures are able to detect closeness beyond algebraic constraints because they can detect two polynomials being close even when a factorization is not yet available. This also a great feature.

\begin{remark}[Closeness beyond algebraicity]
Multivariate polynomial, contrary to univariate ones, can be very close to each other before the existence of an actual factorization let us see this fact. For this reason, for multivariate polynomials is better to explore this notion of closeness line by line, as the measures introduced in Section \ref{measuring} do. In this way, we overcome through a clever trick algebraic restrictions that do not let us see the whole picture immediately.
\end{remark}

This notion of closeness also allow us to work with cofactors without having to know them explicitly as long as they are \textit{far enough} from our target polynomial. More than that, as noted in the remark above, these do not need to exist algebraically for us to make an assessment of closeness through these tools.

\begin{remark}[Factors]
Lately, however, these factors will actually turn out to exist at the points when we reach our target polynomial for algebraic reasons. It is just an advantage being able to proceed without having to think about these factorizations while we work out the topological setting.
\end{remark}

Finally, all these insights combine into forming the object giving us the correct notion of closeness in Definition \ref{rhfd}, which is not a distance and therefore does not provide a topology but still let us perform a topological analysis using that the topologies we want to study are already defined and closely connected to this measure, as noted in Remark \ref{jum}.

The last feature we comment is that of compactness and how it allows us to cover the zero set we want to reach. In particular, Construction \ref{tangcov} allows us to cover all the real zero set using just a finite porduct of linear forms thanks to this compactness. Notice that this avoids the objection someone could make about taking tangent hyperplanes always near to some point in a way that the rest of the real zeros always lie far. We do not allow this, we take tangents that are well distributed in the sense that their opens of (determinantal) deformations allowed end up covering the whole real zero set of our target polynomial. It is this total covering of the power of the (determinantal) deformation we use what let us ultimately reach through a smooth generic deformation our target polynomial in a way that guarantee the conservation of the original (target) RCS although new factors appear far away.

The objections commented above were fundamental in shaping the journey of the proof that we presented here. They show a picture of nuances that had to be correctly addressed before a topological argument could be available and guided us through the process of building such proof through the presented constructions finally in Theorem \ref{glc}.

\printbibliography

@article{helton2007linear,
  title={Linear matrix inequality representation of sets},
  author={Helton, J William and Vinnikov, Victor},
  journal={Communications on Pure and Applied Mathematics: A Journal Issued by the Courant Institute of Mathematical Sciences},
  volume={60},
  number={5},
  pages={654--674},
  year={2007},
  publisher={Wiley Online Library}
}

@article{nuij1968note,
  title={A note on hyperbolic polynomials},
  author={Nuij, Wim},
  journal={Mathematica Scandinavica},
  volume={23},
  number={1},
  pages={69--72},
  year={1968},
  publisher={JSTOR}
}

@article{amini2019spectrahedrality,
  title={Spectrahedrality of hyperbolicity cones of multivariate matching polynomials},
  author={Amini, Nima},
  journal={Journal of Algebraic Combinatorics},
  volume={50},
  number={2},
  pages={165--190},
  year={2019},
  publisher={Springer}
}

@inproceedings{raghavendra2019exponential,
  title={Exponential lower bounds on spectrahedral representations of hyperbolicity cones},
  author={Raghavendra, Prasad and Ryder, Nick and Srivastava, Nikhil and Weitz, Benjamin},
  booktitle={Proceedings of the Thirtieth Annual ACM-SIAM Symposium on Discrete Algorithms},
  pages={2322--2332},
  year={2019},
  organization={SIAM}
}

@MISC {4942793,
    TITLE = {Prove that two multivariate polynomials that are equal at every point over the reals are identical as formal polynomials},
    AUTHOR = {Jack},
    HOWPUBLISHED = {Mathematics Stack Exchange},
    NOTE = {version: 2024-07-07},
    URL = {https://math.stackexchange.com/q/4942793}
}

@book{hausdorff1978grundzuge,
  title={Grundzuge der mengenlehre},
  author={Hausdorff, Felix},
  volume={61},
  year={1978},
  publisher={American Mathematical Soc.}
}

@article{frechet1906quelques,
  title={Sur quelques points du calcul fonctionnel},
  author={Fr{\'e}chet, Maurice},
  year={1906}
}

@article{pompeiu1905continuite,
  title={Sur la continuit{\'e} des fonctions de variables complexes},
  author={Pompeiu, Dimitrie},
  journal={Annales de la Facult{\'e} des sciences de l'Universit{\'e} de Toulouse pour les sciences math{\'e}matiques et les sciences physiques},
  volume={7},
  number={3},
  pages={265--315},
  year={1905}
}

@article{quarez2012symmetric,
  title={Symmetric determinantal representation of polynomials},
  author={Quarez, Ronan},
  journal={Linear algebra and its applications},
  volume={436},
  number={9},
  pages={3642--3660},
  year={2012},
  publisher={Elsevier}
}

@article{helton2006noncommutative,
  title={Noncommutative convexity arises from linear matrix inequalities},
  author={Helton, J William and McCullough, Scott A and Vinnikov, Victor},
  journal={Journal of Functional Analysis},
  volume={240},
  number={1},
  pages={105--191},
  year={2006},
  publisher={Elsevier}
}

@article{stefan2021short,
  title={A short proof of the symmetric determinantal representation of polynomials},
  author={Stefan, Anthony and Welters, Aaron},
  journal={Linear Algebra and its Applications},
  volume={627},
  pages={80--93},
  year={2021},
  publisher={Elsevier}
}

@article{netzer2015smooth,
  title={Smooth hyperbolicity cones are spectrahedral shadows},
  author={Netzer, Tim and Sanyal, Raman},
  journal={Mathematical Programming},
  volume={153},
  number={1},
  pages={213--221},
  year={2015},
  publisher={Springer}
}
\end{document}